\documentclass{amsart}

\usepackage{latexsym,amsthm,amssymb}

\DeclareMathAlphabet{\mathpzc}{OT1}{pzc}{m}{it}

\usepackage[leqno]{amsmath}

\usepackage{mathrsfs}

\usepackage[all]{xy}

\usepackage{color}
\usepackage{colordvi}

\usepackage{verbatim}

\usepackage{ifthen}
\newboolean{hide.proofs}
\setboolean{hide.proofs}{false}

\newtheorem{theorem}{Theorem}[section]
\newtheorem{lemma}[theorem]{Lemma}
\newtheorem{definition}[theorem]{Definition}
\newtheorem{proposition}[theorem]{Proposition}

\newtheorem{remark}[theorem]{Remark}

\newtheorem{lemma-conjecture}[theorem]{Lemma--Conjecture}

\numberwithin{equation}{theorem}

\renewcommand{\mathcal}{\mathscr}

\newcommand{\SI}{{\mathcal{I}}}

\renewcommand{\mathbb}{\mathbf}

\newcommand{\WY}{\widetilde{Y}}

\title{Deformations of canonical double covers}

\author{Francisco Javier Gallego}
\author{Miguel Gonz\'alez}
\author{\\ Bangere P. Purnaprajna}

\address{Departamento de \'Algebra, Universidad Complutense de Madrid and Instituto de Matem\'atica Interdisciplinar, Madrid, Spain}
\email{gallego@mat.ucm.es}
\address{Departamento de \'Algebra, Universidad Complutense de Madrid, Madrid, Spain}
\email{mgonza@mat.ucm.es}
\address{Department of Mathematics, University of Kansas, Lawrence, Kansas, USA }
\email{purna@math.ku.edu}

\begin{document}

\begin{abstract}
In this paper, we show that if $X$ is a smooth variety of general type of dimension $m \geq 2$, for which its canonical map induces a double cover onto $Y$, where $Y$ is a projective bundle over $\mathbf P^1$, or onto a projective space or onto a quadric hypersurface, embedded by a complete linear series, then the general deformation of the canonical morphism of $X$ again is canonical and again induces a double cover. The second part of the article deals with the existence or non existence of canonical double structures on rational varieties. The negative result  in this article has consequences for the moduli of varieties of general type of arbitrary dimension. The results here show that there is an entire component, that is hyperelliptic in infinitely many moduli spaces of higher dimensional varieties of general type. This is in sharp contrast with the case of curves or surfaces of lower Kodaira dimensions.
\end{abstract}

\maketitle

\section*{Introduction}
The canonical covers of rational varieties such as the projective space, projective bundles over a curve and so on have a ubiquitous presence in the geometry of varieties of general type. They occur in various situations and compellingly in the extremal cases. For example in the case of the surfaces on the Noether line $K_{X}^2=2p_g-4$. This was studied and classified by Horikawa in \cite{Hor2p_g-4}. The higher dimensional analogues of these results have been proved in \cite{Fujita}, \cite{Kobayashi} and more recently in \cite{CPG}. In both these cases it turns out that the varieties of general type on the Noether line are canonical double covers of varieties of minimal degree, that is canonical double covers of certain rational varieties. 
\smallskip

In this article, we study the deformations of canonical morphism of degree $2$ of a variety of general type of arbitrary dimension, when the image of the morphism $Y$ is a rational variety that is either a $\mathbf P^m$, a projective bundle over a rational curve or a hyper--quadric, that is the canonical double covers of these rational varieties. As mentioned above, they have a ubiquitous presence in the geometry of varieties of general type. We prove that the general deformation of the canonical morphism of degree $2$ is again a canonical morphism and is of degree $2$.  Some of the cases regarding deformation results have a overlap with results in \cite{Fujita}, but the methods and proofs are qualitatively different. 
\smallskip

Some things are known about the components of the moduli space of surfaces of general type, still there is a lot to be understood. But it is largely in the nascent stage for the moduli of varieties of general type of higher dimension. The result of this article has implications to the moduli of varieties of general type. In \cite{MP} and \cite{Hirzebruch}, we constructed infinitely many moduli spaces with ``hyperelliptic'' components. There were moduli spaces in which there was a hyperelliptic component, as well as a locus in another component which parametrized surfaces of general type whose canonical morphism was of degree $2$, but whose general deformations was a $1:1$ map.  The results of this article does show the existence of the hyperelliptic component in infinitely many moduli of varieties of general type of arbitrary dimension. In \cite{Hirzebruch} it was shown that the double canonical covers of rational surfaces, could have invariants whose $K_{X}^2>3p_g-7$. This is the Castelnuovo bound, and any $|K_X|$ that is birational has to satisfy this inequality. In higher dimensions, there is no such bound to the best of our knowledge. But it can be shown that, for instance in threefolds, double canonical covers can have $K_{X}^3>3p_g-7$. There is indeed a correlation between numerology  in various dimensions. It has been shown in a very recent work of Chen, Gallego and Purnaprajna \cite{CPG} that the ``Noether faces'' for higher dimensions are $K_{X}^m=2(p_g-m)$, and that this determines the precise analogues of Horikawa's results for surfaces in higher dimensions. So conjecturally at least, one should expect similar results for birationality of $K_{X}$ in higher dimensions, as Castelnuovo's results are for curves. 

\smallskip

The higher multiplicity structures on smooth varieties occur very naturally in algebraic  geometry, especially in degenerations or deformation of smooth varieties. They are also of significant interest in their own right. They provide a way of constructing, at least theoretically, smooth varieties with given invariants or canonical varieties, that is varieties of general type whose canonical morphism is a birational morphism. For the case of algebraic surfaces, this is a question of Enriques that has evoked much interest among geometers. So a natural question is to explore under what circumstances, do canonical double structures occur on simpler varieties, for example on rational varieties. In Section $2$ we explore the relation between double covers and multiplicity
$2$ ropes or, more precisely, into the relation between the deformation theory of double
covers and the existence or non--existence of multiplicity $2$ ropes. To do so we will
focus on the case of varieties of general type of dimension $m \geq 2 \,$ that are canonical double covers of either a smooth projective bundle $Y=\mathbb P(E)$ (of dimension $m$) over $\mathbf P^1$ or $\mathbf P^m$ or a hyperquadric $\mathbf Q_m$ or $\mathbb Q_2$, and canonically embedded carpets. This relation has been previously studied for the case of the canonical morphism of hyperelliptic curves and canonical ribbons (see~\cite{Fong}) or for
the case of $K3$ double covers of surfaces of minimal degree and Enriques
surfaces and $K3$ carpets or for the case of $\mathbf P^2$ or a Hirzebruch
surface (see~\cite{GPcarpets}, \cite{GGP2} and \cite{Hirzebruch}). In all
of these cases, the following correlation occurs: whenever a double cover $X$ appears as
limit of projective embeddings, then an embedded multiplicity $2$ rope with same invariants as $X$ appears, so as to bear witness to the event of an embedding degenerating to a degree $2$ morphism.
On the contrary, when the double cover cannot be obtained as limit of embeddings, no multiplicity $2$ ropes come into being. The results here show the situation that always happens for the above, that is the case of curves and surfaces of Kodaira dimension zero, does not happen generally for canonical double covers. That is those varieties of general type, whose canonical map admits a double cover of a rational variety. It was shown in \cite{MP} that there are instances where it does happen. But the results in this article, prove that generically it does not happen. This negative result has consequences for the moduli of varieties of general type of arbitrary dimension. The results here show that there is an entire component, that is hyperelliptic. This is in sharp contrast with the case of curves or surfaces of lower Kodaira dimensions. 

\smallskip

{\bf Acknowledgements}: The first and the second author were partially supported by grants MTM2009-06964 and MTM2012-32670 and by the UCM research group 910772. The first author also thanks the Department of Mathematics of the University of Kansas for its hospitality. The third author thanks the National Science foundation (Grant Number: 1206434) for the partial support of this work.

\section{Deformations of a canonical double cover of a projective bundle over $\mathbf P^1$ or a projective space or a quadric hypersurface embedded by a complete linear series}
We work over $\mathbb C$. We will use the following 

\noindent
{\bf Set--up and notation:} 
\begin{enumerate}
\item $X$ is a smooth variety of general type of dimension $m \geq 2$
with ample and base--point--free canonical bundle and canonical map $\varphi :X \to \mathbb P ^N$ of degree $2$. 
\item The image $Y$ of $\varphi$ is one of the following varieties:
	\begin{enumerate}
		\item a smooth projective bundle $Y=\mathbb P(E)$ (of dimension $m$) over $\mathbf P^1$ (where $E=\mathcal O_{\mathbb P^1} \oplus \mathcal O_{\mathbb P^1}(-e_1) \oplus \cdots \oplus \mathcal O_{\mathbb P^1}(-e_{m-1})$ with $0 \leq e_1\leq \cdots \leq e_{m-1}$) embedded, as a non--degenerate variety, in $\mathbb P ^N $ by the complete linear series of a very ample divisor $aT+bF$, for which $T$ is the divisor on $Y$ such
		that $\mathcal O_Y(T) = \mathcal O_{\mathbf P(E)}(1)$ and $F=\mathbb P^{m-1}$ is a fiber, or
		\item projective space $Y=\mathbb P^m$ embedded, as a non--degenerate variety, in $\mathbb P^N$ by the complete linear series of $\mathcal O_{\mathbb P^m}(d)$, for which $d \geq 1$,
		or
		\item a smooth quadric hypersurface $Y=\mathbb Q_m$ in $\mathbb P^{m+1}$ embedded, as a non--degenerate variety, in $\mathbb P^N$ by the complete linear series of $\mathcal O_{\mathbb Q^m}(d)$ for which $d \geq 1$, when $m \geq 3$, or by the complete linear series of $\mathcal O_{\mathbb Q_2}(a, b)$, for which $a \geq 1, b \geq 1$, when $m=2$ .
		
	\end{enumerate}
\item Let $i: Y \hookrightarrow \mathbb P^N$ denote the embedding and let $\pi: X \to Y$ denote the induced double cover such that $\varphi=i \circ \pi$.
\end{enumerate}
\begin{remark}\label{Seshadri.remark}{\rm Under the conditions of our set--up, $\varphi$ factors through a double cover $\pi: X \longrightarrow Y$ whose trace--zero module is a
line bundle $\mathcal E$. In addition, $\varphi$ is the canonical map of $X$ if, and only if $\mathcal E =\omega_Y(-1)$.
We will call $\varphi$ a canonical double cover.

\smallskip
\noindent
Observe that, since $p_g(Y)=0$, the pullback of the complete linear series of $\mathcal O_Y(1)$ is the complete canonical series of $X$.

\smallskip
\noindent
Observe, also, that there exists a smooth double cover $\pi: X \to Y$ with trace--zero module $\omega_Y(-1)$ if, and only if, the linear system $|-2(\omega_Y(-1))|$ contains a smooth member.}
\end{remark}

\begin{remark}
{\rm Under the conditions of our set--up, numerical conditions which are sufficient for $\pi$ to exist are:
\begin{enumerate}
\item In the case $Y=\mathbb P(E)$, with $m \geq 3$
	\begin{enumerate}
		\item $a \geq 1$ and $b \geq a e_{m-1}+1$ (for $\mathcal O_Y(1)=\mathcal O_{\mathbb P(E)}(aT+bF)$ very ample), and
		\item $b \geq -(e_1+\cdots+e_{m-2}+(1-m-a)e_{m-1}+2)$ (for $|-2(\omega_Y(-1))|$ base--point--free).
		\end{enumerate}
		
		\noindent
		In the case $m=2$ of a Hirzebruch surface $Y=\mathbb F_e$
			\begin{enumerate}
				\item $a \geq 1$ and $b \geq a e+1$ (for $\mathcal O_Y(1)=\mathcal O_{\mathbb F_e}(aT+bF)$ very ample), and
				\item $e \leq 3$, or ($b-ae \geq e -2$ or $b-ae=\frac{1}{2}e-2$), if $e \geq 4$ (for $|-2(\omega_Y(-1))|$ to contain a smooth member, see \cite[Lemma 1.6]{Hirzebruch}).
				\end{enumerate}
\item In the case $Y=\mathbb P^m$, with $m \geq 2$, the only condition  $d \geq 1$.

\item In the case $Y=\mathbb Q_m$, with $m \geq 3$, the only condition $d \geq 1$, and in the case $Y=\mathbb Q_2$ the only condition $a, b \geq 1$.
\end{enumerate}}
\end{remark}

\smallskip
\noindent
We start with a version for higher dimensional varieties of \cite[Lemma 3.9]{MP}, where we had explored the situation for the case of an algebraic surface. We now start with a higher dimensional exploration of the results for surfaces. This technical lemma will be a starting point of this study.
\begin{lemma}\label{isom.Hom.Ext}
Let $S$ be a smooth variety of dimension $m \geq 3$, embedded in $\mathbf P^N$, let $\mathcal J$ be the ideal sheaf of $S$ in $\mathbf P^N$  and consider the connecting homomorphism
\begin{equation*}
\mathrm{Hom}(\mathcal J/\mathcal J^2,\omega_{S}(-1)) \overset{\delta} \longrightarrow \mathrm{Ext}^1(\Omega_{S},\omega_{S}(-1)).
\end{equation*}
\begin{enumerate}
\item If $p_g(S)=0$ and $h^{m-1}(\mathcal O_{S}(1))=0$, then $\delta$ is injective;
\item if $h^{m-1}(\mathcal O_{S})=h^{m-2}(\mathcal O_{S}(1))=0$,
then $\delta$ is surjective.
\end{enumerate}
\end{lemma}

\begin{proof}
The map $\delta$ fits in the following exact sequence:
\begin{equation*}
\mathrm{Hom}(\Omega_{\mathbf P^N}|_S,\omega_S(-1)) \longrightarrow \mathrm{Hom}(\mathcal J/\mathcal J^2,\omega_S(-1)) \overset{\delta} \longrightarrow \mathrm{Ext}^1(\Omega_S,\omega_S(-1))  \longrightarrow \mathrm{Ext}^1(\Omega_{\mathbf P^N}|_S,\omega_S(-1)).
\end{equation*}
Using the restriction to $S$ of the Euler sequence we get
\begin{equation*}
 \mathrm{Hom}(\mathcal O_S^{N+1}(-1),\omega_S(-1)) \longrightarrow \mathrm{Hom}(\Omega_{\mathbf P^N}|_S,\omega_S(-1)) \longrightarrow \mathrm{Ext}^1(\mathcal O_S,\omega_S(-1)).
 \end{equation*}
 Thus, if $p_g(S)=h^{m-1}(\mathcal O_{S}(1))=0$
 then $\mathrm{Hom}(\Omega_{\mathbf P^N}|_S,\omega_S(-1))$ vanishes so $\delta$ is injective. On the other hand, from the Euler sequence we also obtain
 \begin{equation*}
\mathrm{Ext}^1(\mathcal O_S^{N+1}(-1),\omega_S(-1)) \to \mathrm{Ext}^1(\Omega_{\mathbf P^N}|_S,\omega_S(-1)) \to \mathrm{Ext}^2(\mathcal O_S,\omega_S(-1)).
\end{equation*}
Then, if $h^{m-1}(\mathcal O_{S})=h^{m-2}(\mathcal O_{S}(1))=0$, $\mathrm{Ext}^1(\Omega_{\mathbf P^N}|_S,\omega_S(-1))$ vanishes and $\delta$ is surjective.
\end{proof}

\noindent
The next Propositions~\ref{hom=0}, \ref{2.hom=0}, \ref{3.hom=0}, are technical results that are crucial for the proof of Theorem~\ref{def}. 
 
\begin{proposition}\label{hom=0}
Let $Y$ be a smooth projective bundle over $\mathbf P^1$ of dimension $m \geq 2$ embedded, as a non--degenerate variety, in $\mathbf P^N$ by a complete linear series and let $\mathcal I$ be its ideal sheaf. Then $\mathrm{Hom}(\mathcal I/\mathcal I^2,\omega_Y(-1))=0$.  
\end{proposition}

\begin{proof}
If $m=2$, the claim was proved in~\cite[Proposition 1.7]{Hirzebruch}. Now let $Y=\mathbf P(E)$ over $\mathbf P^1$, embedded, as a non--degenerate variety, in projective space $\mathbf P^N$ by a complete linear series and assume $m\geq 3$. Then $Y$ satisfies the hypothesis in (1) and (2) of Lemma~\ref{isom.Hom.Ext}, so it suffices to prove the vanishing of $\mathrm{Ext}^1(\Omega_{Y},\omega_{Y}(-1))=H^1(\mathcal T_Y \otimes \omega_Y(-1))$.

\noindent Having in account the sequence
  \begin{equation*}
 0 \longrightarrow \mathcal T_{Y/\mathbf P^1} \longrightarrow \mathcal T_Y
\longrightarrow  p^*\mathcal T_{\mathbf
P^1} \longrightarrow 0
\end{equation*}
and the dual of the relative Euler sequence
\begin{equation*}
 0 \longrightarrow \mathcal O_{\mathbf P(E)} \longrightarrow p^*E^* \otimes
\mathcal O_{\mathbf P(E)}(1) \longrightarrow  \mathcal T_{Y/\mathbf P^1}
\longrightarrow 0, 
\end{equation*}
we only need to check the vanishings of $H^1(p^*\mathcal T_{\mathbf P^1} \otimes \omega_Y(-1))$, $H^2(\omega_Y(-1))$ and $ H^1(p^*E^* \otimes
\mathcal O_{\mathbf P(E)}(1) \otimes \omega_Y(-1))$. By Serre duality this is equivalent to checking the vanishings of $H^{m-1}(p^*\omega_{\mathbf P^1} \otimes \mathcal O_Y(1))$, $H^{m-2}(\mathcal O_Y(1))$ and $H^{m-1}(p^*E \otimes \mathcal O_{\mathbf P(E)}(-1) \otimes \mathcal O_Y(1))$. 

\smallskip

\noindent
We prove that in the three cases the
cohomology groups on $Y$ are isomorphic to the cohomology groups of the
push--down of the sheaves to $\mathbf P^1$. For this it suffices to see that
the higher cohomology of the restriction of the three sheaves on $Y$ to a fiber
$F$ (which is isomorphic to $\mathbf P^{m-1}$) vanishes. In the first two
cases  
this follows because we are restricting a line bundle, so the intermediate
cohomology would vanish, and since $\mathcal O_Y(1)$ is very ample, so would the topmost
cohomology. In the third case the intermediate cohomology vanishes by the same
reason as before and the topmost cohomology vanishes also because the
restriction of $\mathcal O_{\mathbf P(E)}(-1) \otimes \mathcal O_Y(1)$ to $F$ has
non negative degree. Finally, we have to compute cohomology groups on $\mathbf P^1$.
Except for $H^{m-2}(p_*\mathcal O_Y(1))$ when $m=3$, these cohomology groups
vanish by dimension reasons. If $m=3$, then $H^1(p_*\mathcal O_Y(1))=0$, for
$p_*\mathcal O_Y(1)$ is a direct sum of line bundles of positive degree over
$\mathbf P^1$. 
\end{proof}

\noindent
The next two results are the analogues of the above result for projective space and the smooth quadric, respectively.

\begin{proposition}\label{2.hom=0}
Let $Y=\mathbb P^m$, with $m \geq 2$, embedded, as a non--degenerate variety, in $\mathbf P^N$ by the complete linear series of $\mathcal O_{\mathbb P^m}(d)$, for which $d \geq 1$, and let $\mathcal I$ be its ideal sheaf. Then $\mathrm{Hom}(\mathcal I/\mathcal I^2,\omega_Y(-1))=0$.
\end{proposition}
\begin{proof}
If $m=2$, the claim was proved in~\cite[Proposition 1.7]{Hirzebruch}. Now assume $m\geq 3$. Then $Y=\mathbb P^m$ satisfies the hypothesis in (1) and (2) of Lemma~\ref{isom.Hom.Ext}, so it suffices to prove the vanishing of $\mathrm{Ext}^1(\Omega_{Y},\omega_{Y}(-1))=H^1(\mathcal T_Y \otimes \omega_Y(-1))$.
The proof of $H^1(\mathcal T_Y \otimes \omega_Y(-1))=0$ is straightforward using the Euler sequence in $\mathbb P^m$. 
\end{proof}

\begin{proposition}\label{3.hom=0}
Let $Y=\mathbb Q_m$, with $m \geq 2$, embedded, as a non--degenerate variety, in $\mathbf P^N$ by the complete linear series of $\mathcal O_{\mathbb Q_m}(d)$, for which $d \geq 1$, when $m \geq 3$, or by the complete linear series of $\mathcal O_{\mathbb Q_2}(a, b)$, for which $a \geq 1, b \geq 1$, when $m=2$, and let $\mathcal I$ be its ideal sheaf. Then $\mathrm{Hom}(\mathcal I/\mathcal I^2,\omega_Y(-1))=0$.
\end{proposition}
\begin{proof}
Let $Y=\mathbb Q_m$, embedded in projective space $\mathbf P^N$ by a complete linear series and assume $m \geq 3$. Then $Y$ satisfies the hypothesis in (1) and (2) of Lemma~\ref{isom.Hom.Ext}, so it suffices to prove the vanishing of $\mathrm{Ext}^1(\Omega_{Y},\omega_{Y}(-1))=H^1(\mathcal T_Y \otimes \omega_Y(-1))$. 
The proof of $H^1(\mathcal T_Y \otimes \omega_Y(-1))=0$ is also straightforward using the Euler sequence in $\mathbb P^{m+1}$ restricted to $\mathbb Q_m$ and the normal sequence of $\mathbb Q_m$ in $\mathbb P^{m+1}$.

\smallskip
\noindent
Now we prove the proposition for $Y=\mathbf Q_2$ embedded by $\mathcal
O_{Y}(1)=\mathcal O_{\mathbf Q_2}(a, b)$. In this occasion $p_g(Y)=0$,
$h^1(\mathcal O_Y(1))=0$,
$q(Y)=0$
and $Y$ is embedded by a complete series, so
by \cite[Lemma 3.9]{MP}
it suffices to see that $\mathrm{Ext}^1 (\Omega_Y, \omega_Y(-1))=0$.
For this we use the Euler sequence of $\mathbf P^3$ restricted to $\mathbb Q_2$
\begin{equation}\label{euler.P2}
0 \longrightarrow {\Omega_{\mathbf P^3}}_{\mid \mathbb Q_2} \longrightarrow H^0(\mathcal O_{\mathbf P^3}(1)) \otimes  \mathcal O_{\mathbf Q_2}(-1, -1) \longrightarrow \mathcal O_{\mathbf Q_2} \longrightarrow 0.
\end{equation}
To the sequence~\eqref{euler.P2} we apply the functor  $\mathrm{Hom}(-,\omega_Y(-1))$ to obtain the exact sequence
\begin{equation*}
0 \longrightarrow \mathrm{Ext}^1({\Omega_{\mathbf P^3}}_{\mid \mathbb Q_2},\omega_Y(-1))
\longrightarrow \mathrm{Ext}^2(\mathcal O_{\mathbf Q_2}, \omega_Y(-1))
\longrightarrow  \mathrm{Ext}^2(H^0(\mathcal O_{\mathbf P^3}(1))  \otimes
\mathcal O_{\mathbf Q_2}(-1, -1), \omega_Y(-1)).
\end{equation*}
Dualizing we get
\begin{equation*}
H^0(\mathcal O_{\mathbf P^3}(1)) \otimes  H^0(\mathcal O_{\mathbf Q_2}(a-1, b-1))  \overset{\alpha} \longrightarrow H^0(\mathcal O_{\mathbf Q_2}(a, b)) \longrightarrow \mathrm{Ext}^1({\Omega_{\mathbf P^3}}_{\mid \mathbb Q_2}, \omega_Y(-1))^{\vee} \longrightarrow 0.
\end{equation*}
Now the multiplication map $\alpha$ is surjective if $a, b \geq 1$, so
\begin{equation}\label{ext.van.P2}
 \mathrm{Ext}^1({\Omega_{\mathbf P^3}}_{\mid \mathbb Q_2}, \omega_Y(-1))=0.
\end{equation}
Now we we apply $\mathrm{Hom}(-,\omega_Y(-1))$ to the conormal sequence of $Y=\mathbb Q_2$
\begin{equation*}\label{conormal.Q2}
0 \longrightarrow \mathcal O_{\mathbf Q_2}(-2, -2) \longrightarrow {\Omega_{\mathbf P^3}}_{\mid \mathbb Q_2} \longrightarrow \Omega_{\mathbb Q_2}\longrightarrow 0.
\end{equation*}
to obtain the exact sequence
\begin{equation*}
H^0(\mathcal O_{\mathbf Q_2}(-a, -b))  \longrightarrow \mathrm{Ext}^1(\Omega_Y, \omega_Y(-1)) \longrightarrow \mathrm{Ext}^1({\Omega_{\mathbf P^3}}_{\mid \mathbb Q_2},\omega_Y(-1)).
\end{equation*}
Since $a, b\geq 1$ we have $H^0(\mathcal O_{\mathbf Q_2}(-a, -b))=0$ so, with \eqref{ext.van.P2}, we obtain $\mathrm{Ext}^1(\Omega_Y, \omega_Y(-1))=0$.
\end{proof}

\noindent
Now we are ready to prove the central result of this paper. 

\begin{theorem}\label{def}
Let $X, \varphi$ and $Y$ be as in the set--up. Then any general deformation of $\varphi$ is a canonical map and a finite morphism of degree $2$ onto its image, which is a
deformation of $Y$.
\end{theorem}

\begin{proof}
We are assuming $\varphi : X \to \mathbf P^N$ is the canonical map of $X$ which factors $\varphi =i \circ \pi$, for which $\pi: X \to Y$ is a double cover onto $Y$, which fits in one of the three cases in our set--up  and the embedding $i: Y \hookrightarrow \mathbf P^N$ is non--degenerate and given by a complete linear series as in the set--up.
By Propositions~\ref{hom=0}, \ref{2.hom=0}, \ref{3.hom=0}, we know that $H^1(\mathcal T_Y \otimes \omega_Y(-1))=0$.

\noindent
Let $(\mathfrak X,\Phi)$ be a deformation of $(X,\varphi)$ over a smooth curve $T$ and let
$(X',\varphi')$ be a general member of $(\mathfrak X,\Phi)$.

\noindent
Since $H^1(\mathcal T_Y \otimes \mathcal  \omega_Y(-1))=0$, then \cite[Corollary 1.11]{Wehler}  (or \cite[Theorem 8.1]{Hor3}) tells us that
$\mathfrak X$ can be realized as a covering $\Pi : \mathfrak X \to \mathfrak Y $ (which, after shrinking $T$, is finite, surjective and of degree $2$) of a deformation $\mathfrak Y$ of $Y$.
Now, since $H^1(\mathcal T_{\mathbf P^N}|_Y)=0$ (as can be easily checked in all cases of our set--up), any deformation $\mathfrak Y$ of $Y$ can be realized as a deformation $(\mathfrak Y, \mathfrak i)$ of $(Y,i)$ (see \cite[Theorem 8.1]{Hor3}). Since $i$ is an embedding, we can assume, after shrinking $T$, that $\mathfrak i$ is a relative embedding of $\mathfrak Y$ in $\mathbf P^N_T$.
Then $\mathfrak i \circ \Pi=\Phi'$ is a deformation of
$\varphi$. Now we saw in \cite[Lemma 2.4]{MP} that the only deformation of $\varphi$
is the (relative) canonical morphism, so $\varphi'$ is the canonical morphism
of $X'$ and, after shrinking $T$, we have $\Phi=\mathfrak i \circ \Pi$. 
\end{proof}

\section{Non--existence of higher dimensional carpets}
In this section we link two themes: on the one hand, the deformation theory of double covers $\varphi$ of a smooth variety $Y$, embedded in the projective space $\mathbf P^N$ and, on the other hand, the existence or non--existence of carpets supported on $Y$, also embedded in $\mathbf P^N$. We will focus on the case of canonical double covers (see Remark~\ref{Seshadri.remark}) and canonically embedded carpets (see Definition~\ref{def.canonical.carpet}) and restrict our attention to when $Y$ is either a smooth projective bundle $Y=\mathbb P(E)$ (of dimension $m \geq 2$) over $\mathbf P^1$ or $\mathbf P^m$ or a hyper--quadric $\mathbf Q_m$ or $\mathbb Q_2$ embedded by a complete linear series in $\mathbf P^N$, as in the set--up. The study of both themes will be addressed by looking at
the same cohomology group on $Y$, namely $H^0(\mathcal N_{Y,\mathbf
P^N} \otimes \mathcal \omega_Y(-1))$. To start seeing the reason for this relation between double covers and carpets, consider a morphism $\varphi$ from a smooth irreducible variety $X$ to $\mathbf P^N$ such that $\varphi$ factors as $\varphi = i \circ \pi$, where $\pi$ is a
finite, double cover of $Y$ and $i$ embedds $Y$ in $\mathbf P^N$. The morphism
$\pi$ is flat and its trace--zero module $\mathcal E$ is a line bundle. We then focus on
the group $H^0(\mathcal N_\varphi)$, which parameterizes first order
infinitesimal deformations of $\varphi$, and on the group $H^0(\mathcal
N_{i(Y),\mathbf P^N} \otimes \mathcal E)$ which, according to
\cite[Proposition 2.1]{Gon}, parameterizes pairs $(\widetilde Y, \widetilde
i)$, where $\widetilde Y$ is a rope on $Y$ with conormal bundle $\mathcal E$
and $\widetilde i$ is a morphism from $Y$ to $\mathbf P^N$ extending $i$.
The relation between these two groups, which is the relation that links the
deformation theory of $\varphi$ with the existence or non--existence of carpets, is
given by the following result, which holds in wider  generality (it holds for
$X$ and $Y$ smooth irreducible projective varieties of arbitrary dimension and $\pi$ finite morphism of any degree $n \geq 2$ with trace--zero module $\mathcal E$ of rank $n-1$):

\begin{proposition}\label{morphism.miguel} {\rm (\cite[Proposition 3.7]{Gon}).}
Let $X$ be a smooth irreducible variety and let $\varphi$ be
a morphism  from $X$ to $\mathbf P^N$
that factors as $\varphi = i \circ \pi$, where $\pi$ is a
finite cover of a smooth variety $Y$ and $i$ embedds $Y$ in $\mathbf P^N$. Let
$\mathcal E$ be
the trace--zero module of $\pi$ and let $\mathcal I$ be the ideal sheaf of
$i(Y)$ in $\mathbf P^N$.
There exists a homomorphism
\begin{equation*}
 H^0(\mathcal N_\varphi) \overset{\Psi}\longrightarrow
\mathrm{Hom}(\pi^*(\mathcal I/\mathcal I^2), \mathcal O_X),
\end{equation*}
that appears when taking cohomology on the commutative
diagram~\cite[(3.3.2)]{Gon}. Since
\begin{equation*}
\mathrm{Hom}(\pi^*(\mathcal I/\mathcal I^2), \mathcal O_X)=\mathrm{Hom}(\mathcal
I/\mathcal I^2, \pi_*\mathcal O_X)=\mathrm{Hom}(\mathcal I/\mathcal I^2,
\mathcal O_Y) \oplus \mathrm{Hom}(\mathcal I/\mathcal I^2, \mathcal E)
\end{equation*}
the homomorphism $\Psi$ has two components
\begin{eqnarray*}
H^0(\mathcal N_\varphi) & \overset{\Psi_1}  \longrightarrow &
\mathrm{Hom}(\mathcal I/\mathcal I^2, \mathcal O_Y) \cr
H^0(\mathcal N_\varphi) & \overset{\Psi_2}  \longrightarrow &
\mathrm{Hom}(\mathcal I/\mathcal I^2, \mathcal E).
\end{eqnarray*}
\end{proposition}

\noindent
Now we define canonically embedded carpets.
\begin{definition}\label{def.canonical.carpet}
{\rm Let $Y$ be a smooth, irreducible, projective variety of and let $i: Y
\hookrightarrow \mathbf P^N$ be an embedding of $Y$ in $\mathbf P^N$. Let
$\widetilde Y$ a scheme such that
\begin{enumerate}
\item[a)] $(\widetilde Y)_{\mathrm{red}}=Y$;
\item[b)] $\SI_{Y, \widetilde Y}^2=0$;  and
\item[c)] $\SI_{Y, \widetilde Y}$ is a line
bundle on $Y$ (i.e., $\widetilde Y$ is a rope of multiplicity $2$ on
$Y$), called the conormal bundle of $\widetilde Y$.
\end{enumerate}
Let ${\tilde i}: \widetilde Y
\hookrightarrow \mathbf P^N$ be an embedding of $\WY$ in $\mathbf P^N$ that
extends $i$. We say that $\WY$ is canonically embedded by $\tilde i$ or, for
short, that $\tilde i(\WY)$ is a canonical carpet, if the dualizing sheaf
$\omega_{\widetilde Y}$ of $\WY$ is very ample and $\tilde i$ is induced by the
complete linear series of $\omega_{\widetilde Y}$.}
\end{definition}

\noindent
Our next goal is to characterize canonical carpets on $Y$. For this, we characterize the conormal bundles of canonical carpets. In doing so we begin to unfold the relation between canonical double covers and canonical carpets.
\begin{lemma}\label{trace.conormal} Let $\widetilde Y$ be a carpet on $Y$,
canonically embedded in $\mathbf P^N$ by an embedding $\widetilde i$ extending
$i$. Let $\mathcal E$ be the conormal bundle of $\widetilde Y$. Then $\mathcal E$ is isomorphic to $\omega_Y(-1)$.
\end{lemma}

\begin{proof}
This follows from the equality ${\omega_{\widetilde Y}}_{\mid Y}=\mathcal E^{-1}\otimes \omega_Y$ (see \cite[Lemma 1.4]{GGP2}).
\end{proof}

\noindent
Finally  we apply Propositions~\ref{hom=0}, \ref{2.hom=0}, \ref{3.hom=0}, to prove the
non--existence of canonical carpets on either a smooth projective bundle $Y=\mathbb P(E)$ (of dimension $m \geq 2$) over $\mathbf P^1$ or $\mathbf P^m$ or a hyper--quadric $\mathbf Q_m$ or $\mathbb Q_2$ embedded by a complete linear series in $\mathbf P^N$, as in the set--up.

\begin{theorem}\label{nonexist.canonical.structures.cor} Let $Y$ and $i$ be as in the set--up. There are no carpets inside $\mathbf P^N$, supported on $i(Y)$,
canonically embedded in $\mathbf P^N$.
\end{theorem}

\begin{proof}
By Propositions~\ref{hom=0}, \ref{2.hom=0}, \ref{3.hom=0}, we know that $\mathrm{Hom}(\mathcal I_Y/\mathcal I_Y^2, \omega_Y(-1))=0$. Then, the result follows Lemma~\ref{trace.conormal} and~\cite[Proposition 2.1]{Gon}.
\end{proof}

\end{document}